\documentclass[11pt]{amsart}
\usepackage[mathscr]{eucal}
\usepackage{times}
\usepackage{amsmath,amssymb,latexsym,amscd}   
\usepackage{hyperref}
\hypersetup{colorlinks=true,linkcolor=blue,citecolor=red,linktocpage=true}
\usepackage{graphicx}
\usepackage[all,cmtip]{xy}
\usepackage{fancyhdr}
\usepackage{mathrsfs}
\usepackage{upgreek}
\usepackage{enumerate}
\DeclareMathOperator{\Mod}{-Mod}

\DeclareMathOperator{\Ann}{Ann}

\DeclareMathOperator{\spec}{Spec}
\DeclareMathOperator{\pt}{pt}

\DeclareMathOperator{\sob}{sob}
\DeclareMathOperator{\mx}{Max}
\DeclareMathOperator{\prt}{Prt}
\DeclareMathOperator{\Hom}{Hom}
\DeclareMathOperator{\End}{End}

\theoremstyle{plain}
\newtheorem*{theorem*}{Theorem}
\newtheorem{thm}{Theorem}[section]
\newtheorem{cor}[thm]{Corollary}
\newtheorem{lem}[thm]{Lemma}
\newtheorem{prop}[thm]{Proposition}

\theoremstyle{definition}
\newtheorem{dfn}[thm]{Definition}
\newtheorem{obs}[thm]{Remark}
\newtheorem{ej}[thm]{Example}

\newcommand{\sm}{\sigma[M]}
\newcommand{\annm}{Ann_M}

\begin{document}
%\begin{frontmatter}

\title[Attaching topological spaces to a module]{Attaching topological spaces to a module (I): \\ Sobriety and spatiality}

%\address{Instituto de Matem\'aticas, Ciudad Universitaria, UNAM, M\'exico, D. F., 04510, M\'exico.}
%\email{zaldivar@matem.unam.mx}

\author[Medina,\\Morales,\\Sandoval,\\ Zald\'ivar]{Mauricio Medina B\'arcenas\\Lorena Morales Callejas \\ Martha Lizbeth Shaid Sandoval Miranda\\ \'Angel Zald\'ivar Corichi }
%$\cortext[cor1]{Corresponding author}

\begin{abstract}
In this paper we study some frames associated to an $R$-module $M$. We define semiprimitive submodules and we prove that they form an spatial frame canonically isomorphic to the topology of $\mx(M)$. We characterize the soberness of $\mx(M)$ in terms of the point space of that frame. Beside of this, we study the regularity of an spatial frame associated to $M$ given by annihilator conditions.  
\end{abstract}

\maketitle

\section{Introduction}\label{sec1}
For the study of a module categories over a unitary ring one technique is through the examination of complete lattices of submodules, such as in \cite{albu2014topics}, \cite{albu2014osofsky}, \cite{simmonslattice} and \cite{mauquantale}. Following this idea, we give a module-theoretical counterpart  of the ring-theoretical examination developed by  H. Simmons in \cite{simmons1989compact} and \cite{simmonssome}. We observe some point-free topology aspects of certain frames that we construct trough the manuscript and that consideration eventually leads to the construction of some spatial frames, that is, they behave like a topology (\cite{johnstone1986stone}, \cite{simmons2001point} and \cite{picado2012frames}).

In order to obtain these frames and prove their spatiality, we make use of results and techiniques arised from lattice theory  as in \cite{rosenthal1990quantales}, \cite{mauquantale} and \cite{simmonstwo}, and the usual algebraic techniques developed in categories, in particular those for the category $\sigma[M]$ (see \ref{subsecmod}) studied in \cite{wisbauerfoundations}, \cite{BicanPr}, \cite{beachy2002m}. Besides these tools, we will introduce the point-free topology perspective applied to frames provided by modules as in \cite{mauquantale}. 

In \cite{mauquantale}, it was studied the prime spectrum of a module as well as a generalization of quantales, namely quasi-quantales making use of the techniques mentioned above. One of the main results in that article is that the frame of semiprime modules is an spatial frame like in the classical case of commutative ring theory, \cite[Proposition 4.27]{mauquantale}. In this manuscript we continue this study for the case of the space of (fully invariant) maximal submodules of a given module. It should be noted that many of the results will use the theory developed in \cite{mauquantale}.

In this article we study primitive submodules of a module $M$, which are just annihilators of simple modules in $\sigma[M]$, and so will get the frame of semiprimitive submodules which turns out to be spatial. Moreover this one is isomorphic to the frame of open sets of the space of maximal submodules of $M$ with the hull-kernel topology (\cite{simmons2001point}, \cite{rosenthal1990quantales}). An application of the above result is that for a pm-module (\ref{pmmodule}), the space of maximal submodules concides with that of its primitive submodules and the point space of the frame of semiprimitive submodules, see Proposition \ref{pmpt}. As in the case of commutative rings, we study the sobriety of the space of prime submodules and the space of maximal submodules of a given module. For instance, we get necessary and sufficient conditions for the sobriety on the space of maximal submodules.

Since the idea of annihilators of submodules is present in the study of these frames, following the idea of \cite{simmonssome} we introduce the frame $\Psi(M)$ for a module $M$ which is given in terms of annihilators ( see \ref{psiM}) this frame turns to be spatial and we give sufficient conditions for it to be regular.

The interest of the sudy of these aspects comes from the neccesity of triyng to give a module-theoretic counterpart of the ring-theoretic aspect of sheaves representations of \cite{borceux1984sheaf}, \cite{borceux1983recovering} and \cite{borceuxalgebra} and also the localic point of view of these representations given in \cite{simmons1989compact} and \cite{simmons1985sheaf}.

The organization of the paper is as follows: Section \ref{sec1} is this introduction, Section \ref{sec2} provides the necessary (perhaps not in full detail) material that is needed for the reading of the next sections. Section \ref{sec3} is concerned with the first spatial frame associated to a module, the frame of \emph{ semiprimitve submodules}, we examinate its \emph{point space} and we see when it coincides with the \emph{primitive submodules}.
In Section \ref{sec4} it is proved the \emph{sobriety} of the space $\spec(M)$ (the space of prime submodules) and we give necessary and sufficient conditions for $\mx(M)$ to be a sober space.
In section \ref{sec5} we introduce a set of submodules $\Psi(M)$ defined by annihilator conditions. We prove that this set is an spatial frame and we explore its regularity. 

\section{Preliminaries.}\label{sec2}

\subsection{Idiomatic-quantale preliminaries}\label{subsecidiomq}

First we give the idiomatic preliminaries that are necessary for our analysis. 

 \begin{dfn}\label{idiom}
An \emph{idiom} $(A,\leq,\bigvee,\wedge,\bar{1},\underline{0})$ is a upper-continuous and modular lattice, that is, $A$ is a complete lattice that satisfies the following distributive laws: 
\[a\wedge (\bigvee X)=\bigvee\{a\wedge x\mid x\in X\}\leqno({\rm IDL})\]
holds for all $a\in A$  and $X\subseteq A$ directed; and 
\[a\leq b\Rightarrow (a\vee c)\wedge b=a\vee(c\wedge b)\leqno({\rm ML})\]
for all $a,b,c\in A$. 
\end{dfn}
A good account of the many uses of these lattices can be found in \cite{simmonsintroduction} and \cite{simmonscantor}.

Examples of idioms are given by submodules of a module over a ring. First some conventions, in the whole text $R$ will be an associative ring with identity, not necessarily commutative. The word \emph{ideal} will mean two-sided ideal, unless explicitly stated the side (left or right ideal). All modules are unital and left modules. Let $M$ be an $R$-module, a submodule $N$ of $M$ is denoted by $N\leq M$, whereas we write $N<M$ when $N$ is a proper submodule of $M$. Recall that $N\leq M$ is called \emph{fully invariant submodule}, denoted by $N\leq_{fi}M$, if for every endomorphism $f\in End_R(M)$, it follows that $f(N)\subseteq N.$ Denote by 
\[\Lambda(M)=\{N\mid N\leq M\}\]
\[\Lambda^{fi}(M)=\{N\mid N\leq_{fi}M\}\]

It is straightforward to see that $\Lambda(M)$ and $\Lambda^{fi}(M)$ are idioms, in particular the left (right) ideals of the ring $R$, $\Lambda(R)$ constitutes an idiom.

Another important example come from topology, consider any topological space $S$ and denote by $\mathcal{O}(S)$ its topology, then it is know that $\mathcal{O}(S)$ is a complete lattice with satisfies the distributivity laws of definition \ref{idiom}, in fact it satisfies the arbitrary distributivity law not just for directed subsets. This kind of lattices are called \emph{frames}.

\begin{dfn}
A complete lattice $(A,\leq,\bigvee,\wedge,\bar{1},\underline{0})$ is a \emph{frame}, if $A$ satisfies
\[a\wedge (\bigvee X)=\bigvee\{a\wedge x\mid x\in X\}\leqno({\rm FDL})\]
for all $a\in A$ and $X\subseteq A$ any subset.
\end{dfn}

Frames are certain kind of idioms:

\begin{prop}[\cite{simmonsintroduction} Lemma 1.6]\label{disid}
Let $A$ be an idiom, $A$ is a frame if and only if $A$ is a distributive idiom, that is, 
\[a\vee(b\wedge c)=(a\vee b)\wedge (a\vee c)\]
or equivalently 
\[a\wedge(b\vee c)=(a\wedge b)\vee (a\wedge c)\]  
holds for every $a, b, c\in A$.
\end{prop}

Thus idioms are a generalization of frames. There exists other structures that are generalizations of frames.

\begin{dfn}\label{quan}
A complete lattice $A$ is a \emph{quantale} if $A$ has a binary associative operation $\cdot\colon A\times A\rightarrow A$ such that 
\[l\left(\bigvee X\right)=\bigvee\left\{lx\mid x\in X\right\} , \mbox{ and } \leqno({\rm LQ}) \] 
 \[\left(\bigvee X\right)r=\bigvee\left\{xr\mid x\in X\right\}\leqno({\rm RQ})\] 
hold for all $l, r\in A$ and $X\subseteq A$.  
\end{dfn}
See \cite{rosenthal1990quantales} for theory of quantales.
The canonical example of a quantale is provided by $\Lambda(R)$ with $R$ a ring, with the usual product of ideals and as we will see for a module $M$ (with some extra properties) we can give to $\Lambda(M)$ a structure of quantale. By \emph{idiomatic-quantale} we will mean an idiom that is also a quantale. 

One of the main tools for the study of all of these structures are \emph{adjoint situations}.
%, this kind of situation are a particular case of the well-known categorical situation viewing each partial ordered set as a order category and each monote function as a functor. 

\begin{dfn}
A morphism of $\bigvee$-semilattices, $f\colon A\rightarrow B$ is a monotone function that preserves arbitrary suprema, that is, $f[\bigvee X]=\bigvee f[X]$ for all $X\subseteq A$.
\end{dfn} 

We are going to use next fact repetitively.
If the partial ordered set $A$ is a $\bigvee$-semilattice
%, that is, each subset $X$ of $A$ has a supremum in $A$ 
we have the following important fact.

\begin{prop}\label{hullker}
Given any morphism of $\bigvee$-semilattices, $f\colon A\rightarrow B$ there exits  $f_*\colon B\rightarrow  A$ such that 
\[f^{*}(a)\leq b\Leftrightarrow a\leq f_{*}(b)\]
that is, $f$ and $f_*$ form an adjunction 
\[\xymatrix{ A\ar@<-.7ex> @/   _.5pc/[r]_--{f^{*}} &B\ar@<-.7ex>@/ _.5pc/[l]_--{f_{*}} }\]
\end{prop}

This is a particular case of the General Adjoint functor theorem, a proof of this can be found in any standard book of category theory for instance \cite[Theorem 6.3.10]{leinster2014basic}.
 
The following is easy to check.

\begin{prop}[Lemma 3.3, \cite{simmonsintroduction}]\label{idad}
From this situation, we can consider the two compositions \[f_{*}f^{*}\colon A\rightarrow A \text{ and } f^{*}f_{*}\colon B\rightarrow B.\] Then:
\begin{enumerate}[\rm(1)]
\item $f_{*}f^{*}$ is an \emph{inflator}, that is, is a monotone and $a\leq f_{*}f^{*}(a)$ for all $a\in A$.
\item $f^{*}f_{*}$ is a \emph{deflator}, that is, is a monotone and $f^{*}f_{*}(b)\leq b$ for all $b\in B$.
\item $f_{*}f^{*}$ and $f^{*}f_{*}$ are idempotent.
\end{enumerate}
\end{prop}

Now if we are dealing with idioms, and the morphism $f\colon A\rightarrow B$ also satisfies $f(a\wedge b)=f(a)\wedge f(b)$ then from Proposition \ref{idad} we have:

\begin{prop}[Lemma 3.12, \cite{simmonsintroduction}]\label{idadd}
The closure operators $f_{*}f^{*}$ and $f^{*}f_{*}$ satisfies: \[f_{*}f^{*}(a\wedge b)=f_{*}f^{*}(a)\wedge f_{*}f^{*}(b)\] and \[f^{*}f_{*}(a\wedge b)=f^{*}f_{*}(a)\wedge f^{*}f_{*}(b)\]
\end{prop} 

Next we will review the above properties in the context of idioms and quantales.
A good account of all these facts and their proofs are in \cite{simmonsintroduction} and \cite{rosenthal1990quantales}.

\begin{dfn}\label{nuc}
Let $A$ be an idiom. A \emph{nucleus} on $A$ is a function $j\colon A\rightarrow A$ such that:\begin{itemize}
\item[(1)] $j$ is an inflator.

\item[(2)] $j$ is idempotent.

\item[(3)] $j$ is a \emph{prenucleus}, that is, $j(a\wedge b)=j(a)\wedge j(b)$.

\end{itemize}
\end{dfn}

There also exists the quantale counterpart of this notions:

\begin{dfn}\label{quantic}

A \emph{quantic} nucleus $k$ over a quantale $A$ is an idempotent inflator such that $k(a)k(b)\leq k(ab)$ for all $a, b\in A$.
\end{dfn}

As we are dealing with idiomatic-quantales, we require that our nuclei control the product of the quantale and the $\wedge$ in some nice way, thus we have the following improvement,
 
\begin{dfn}\label{quantic2}
A \emph{multiplicative} nucleus, $k\colon A\rightarrow A$ is an idempotent inflator such that $k(a)\wedge k(b)=k(ab)$.
\end{dfn}

There are several uses of these operators in literature for example they provide a relative version of $A$, that is:

\begin{prop}\label{fix}
Let $A$ be an idiomatic-quantale, let $j$ be a nucleus (or a multiplicative nucleus or a quantic nucleus) then consider the set \[A_{j}=\{a\in A\mid j(a)=a\}\] the set of all \emph{fixed points} of $j$. Then
\begin{enumerate}[\rm(1)]
\item If $j$ is a nucleus then $A_{j}$ is an idiom.
\item If $j$ is a quantic nucleus then $A_{j}$ is a quantale.
\item If $j$ is a multiplicative nucleus and the following inequality holds \[j(a)j(a)\leq j(a)\wedge j(b)\] for all $a, b\in A$ then $A_{j}$ is an idiomatic-quantale.
\end{enumerate}
\end{prop}

The deatails about these facts can be found in \cite{simmonsintroduction} and in a more general context in \cite[Propositions 3.9, 3.10]{mauquantale} where the notion of \emph{quasi-quantale} is introduced,	

\begin{dfn}\label{a}
	Let $A$ be a $\bigvee$-semilattice. We say that $A$ is a \emph{quasi-quantale} if it has an associative product $A\times{A}\rightarrow{A}$ such that for all directed subsets $X,Y\subseteq{A}$ and $a\in{A}$: 
	\[(\bigvee X)a=\bigvee\{xa\mid x\in X\},\] and 
	\[a(\bigvee Y)=\bigvee\{ay\mid y\in Y\}.\]
	We say $A$ is a \emph{left (resp. right, resp. bilateral) quasi-quantale} if there exists $e\in{A}$ such that $e(a)=a$ (resp. $(a)e=a$, resp. $e(a)=a=(a)e$) for all $a\in{A}$. 
\end{dfn}

\begin{dfn}\label{d} Let $A$ be a quasi-quantale and $B$ a sub $\bigvee$-semilattice.  We say that $B$ is a {\em subquasi-quantale of $A$} if 
\[(\bigvee{X})a=\bigvee\{xa\mid x\in{X}\}\] and
\[a(\bigvee{Y})=\bigvee\{ay\mid y\in{Y}\},\]
for all directed subsets $X,Y\subseteq{B}$ and $a\in{B}$.
\end{dfn}

In that investigation the authors use this kind of operators to give a radical theory for modules and they observed that this theory is related with some spatial properties of these structures, let us recall some of that material.

As we saw any topological space $S$ determines a frame, its topology $\mathcal{O}(S)$, this defines a functor from the category of topological spaces into the category of frames $\mathcal{O}(\_)\colon Top\rightarrow Frm$. There exists a functor in the other direction:

\begin{dfn}\label{point}
Let $A$ be a frame. An element $p\in A$ is a \emph{point} or a $\wedge$\emph{-irreducible}
if $p\neq 1$ and $a\wedge b\leq p \Rightarrow a\leq p \text{ or } b\leq p.$
\end{dfn}
Denote by $\pt(A)$ the set of all points of $A$.

We can topologize this set as follows, for each $a\in A$ define: \[p\in U_{A}(a)\Leftrightarrow a\nleq p\] for $p\in\pt(A)$. The collection $\mathcal{O}\pt(A)=\{U_{A}(a)\mid a\in A\}$ constitutes a topology for $\pt(A)$. We have a frame morphism \[U_{A}\colon A\rightarrow \mathcal{O}\pt(A)\]
that determines a nucleus on $A$ by Proposition \ref{hullker}, this nucleus or the adjoint situation is called the \emph{hull-kernel adjunction}. With this, the frame $A$ is \emph{spatial} if $U_{A}$ is an injective morphism (hence an isomorphism). 

With some work one can prove that this defines a functor $\pt(\_)\colon Frm\rightarrow Top$ in such way that the pair \vspace{-10pt}
\[\xymatrix{ &Top\ar @/^/[r]^{\mathcal{O}(\_)} & Frm \ar @/^/[l]^{\pt(\_)}} \] 
form a adjoinction. For details see \cite{johnstone1986stone}, \cite{simmonspoint} and \cite{picado2012frames}.

For a (quasi-)quantale there exists a similar way to produce a space. 

\begin{dfn}\label{e}
Let $A$ be a quasi-quantale and $B$ a subquasi-quantale of $A$. An element $1\neq{p}\in{A}$ is a {\em prime element relative to $B$} if whenever $ab\leq{p}$ with $a,b\in{B}$ then $a\leq{p}$ or $b\leq{p}$.
\end{dfn}

%\begin{dfn}\label{primqu}
%Let $A$ be a quantale and $p\in A$. It is said $p$ is a \emph{prime element} if $p\neq 1$ and 
%$ab\leq p \Rightarrow a\leq p \text{ or } b\leq p.$
%\end{dfn}

Denote $Spec_B(A):=\{p\mid p\text{ is prime in}A\text{ relative to }B  \}$. We can topologize it imitating the process for $\pt(A)$, and obtain the \emph{Zariski-like topology} for $Spec_B(A)$.

\begin{thm}
Let $A$ be a quasiquantale,  $B$ a sub-quasiquantal of $A.$ Then, the following conditions hold.
\begin{enumerate}[\rm(1)]
\item $Spec_B(A)$ is  a topological space where the closed sets are given by  $V(b)=\{p\in Spec_B(A)\mid b\leq p\}$ with  $b\in B.$ Duality, the open sets are give by  $U(b)=\{p\in Spec_B(A)\mid b\not\leq p\}.$ \cite[Proposition 3.18]{mauquantale}
\item For every $b\in B,$ $\mu(b)$ is the greatest element in  $B$ such that  $\mu(b)\leq \bigwedge\{p\in Spec_B(A)\mid p\in V(b)\}.$ \cite[Proposition 3.20]{mauquantale}
\item The $\bigvee-$semilattice morphism $U:B\to \mathcal{O}(Spec_B(A))$  has  a right adjoint $\mathcal{U}_{\ast}: \mathcal{O}(Spec_B(A))\to B$ given by  $\mathcal{U}_{\ast}(W):=\bigvee \{b\in B\mid U(b)\subseteq W\}$ and the  composition $\mu=U_{\ast}\circ U:B\to B$  is a  multiplicative nucleus.  \cite[Proposition 3.21]{mauquantale} 
\item $ B_{\mu}:=\{x\in B\mid \mu(x)=x \}$ is upper continuous. \cite[Corollary 3.22]{mauquantale}
\item If in addition, $B$ is a quasi-quantale such that  $(\bigvee X)a=\bigvee\{xa\mid x\in X\},$ for every  $X\subseteq A$ and  $a\in A,$  then $B_{\mu}$ is a frame. \cite[Corollary 3.11]{mauquantale}
\end{enumerate}
\end{thm}

The following statements are important to prove Theorem \ref{PSI}, they are hard to find in the literature thus for convenience of the reader we provide their proofs.

Our prototype of idiomatic-quantale is $\Lambda(R).$ Note that this lattice is also compactly generated:

\begin{dfn}\label{compac}  Let $A$ be an idiomatic-quantale: 
\begin{enumerate}[\rm(1)]
\item An element $c\in A$ is \emph{compact} \[c\leq \bigvee X\Rightarrow (\exists x\in X)[c\leq x]\] for each $X$ directed subset of $A$. This is equivalent to consider any subset $X$ of $A$ such that $c\leq \bigvee X$ then $c\leq\bigvee F$ for a $F\subseteq X$ finite.
\item $A$ is \emph{compactly generated} if for each $a\in A$ we have $a=\bigvee C$ for some set $C$ of compact elements.
\end{enumerate}
\end{dfn}

Now consider a sublattice $\Lambda$ of $A$ and a compact element $m\in A$. Set
\[\mathcal{X} (m)=\{x\in \Lambda\mid m\not\leq x\}.\]
Observe that this subset depends on $\Lambda$, and if $m=0$ then $\mathcal{X}(m)=\emptyset.$

\begin{prop}\label{compact elem 1}
Let $A$ be an idiomatic-quantale, $\Lambda$ a sublattice of $A$ and $m\in A$ a compact element. Then,
\begin{enumerate}[\rm(1)]
\item The family $\mathcal{X}(m)$ is closed under directed unions.
\item Each member of $\mathcal{X}(m)$ is less than or equal to a maximal member of $\mathcal{X}(m).$
\end{enumerate}
\end{prop}
\begin{proof}
$\rm(1)$ Let $D$ be a directed subset of $\mathcal{X}(m)$. Suppose that $m\leq \bigvee D$. Since $m$ is a compact element, we have $m\leq \bigvee F$ for some finite family $F\subseteq D$. Thus there exists $d\in D$ such that $m\leq \bigvee F\leq d$. This gives $d\notin\mathcal{X}(m)$ which is a contradiction.

$\rm (2)$ This is an application of (1) and Zorn's Lemma.
\end{proof}

We know that if $A$ is an idiom then $A$ is a frame provided it is a distributive lattice (Proposition \ref{disid}), there exists many non-distributive idioms, but we can subtract the distributive part of a idiom:

\begin{dfn}\label{disd}
Let $A$ be an idiom. An element $a\in A$ is \emph{distributive} if it satisfies the following equivalent conditions:
\begin{enumerate}[\rm(1)]
\item $(\forall x,y\in A[a\vee(x\wedge y)=(a\vee x)\wedge(a\vee y)])$.

\item $(\forall x,y\in A[a\wedge(x\vee y)=(a\wedge x)\vee(a\wedge y)])$.
\end{enumerate}

\end{dfn}

Denote by $F(A)$ the set of all distributive elements of $A$, it is easy to see that $F(A)$ is a sub-idiom of $A$ which is distributive and thus it is a frame.

\begin{prop}\label{pointid}
Let $A$ be an idiomatic-quantale, $\Lambda$ a sub-idiom of $F(A)$ and $m\in A$ a compact element. Then each maximal member of $\mathcal{X}(m)$ is a point of $\Lambda.$
\end{prop}
\begin{proof}
Let $p$ be a maximal element of $\mathcal{X}(m)$. Since $m\not\leq p$ then $p\neq 1.$ Now, suppose $a,b\in \Lambda$ are elements such that $a\not\leq p$ and $b\not\leq p$. We shall show that $a\wedge b\not\leq p$. We have that $p<a\vee p$ and $p<b\vee p$ since $a\not\leq p$ and $b\not\leq p$. Thus $a\vee p\notin \mathcal{X}(m)$ and $b\vee p \notin \mathcal{X}(m)$ by the maximality of $p$. This gives 
\[m\leq (a\vee p)\wedge (b\vee p)=(a\wedge b)\vee p \]
where the right hand holds since $\Lambda$ is distributive. Finally, since $m\not\leq p$, then $a\wedge b\not\leq p$, as required.
\end{proof}

Let $\Lambda$ as above. Denote by $\pt(\Lambda)$ the set of points of $\Lambda$ then it is not empty. Let $S\subseteq \pt(\Lambda)$. For each $a\in A$ we define the following subset of $S$,
\[p\in d(a)\Leftrightarrow a\nleq p.\] 
Therefore we can topologize $S$ with the family $\mathcal{O}(S)=\{d(a)\mid a\in A\}$, thus we have an idiom epimorphism \[d\colon\Lambda\rightarrow\mathcal{O}(S)\] with this we can prove the following which has an important consequence and illustrate the spatial behaviour of distributive lattices of compactly generated idioms. 

\begin{thm}\label{spatiid}
Let $A$ be an idiomatic-quantale compactly generated and $\Lambda$ a sub-idiom of $F(A)$. Then the associated indexing morphism 
\[d\colon\Lambda\rightarrow \mathcal{O}\pt(\Lambda)\]
is injective. Hence $\Lambda$ is spatial.
\end{thm}
\begin{proof}
We show that $d(a)\subseteq d(b)$ implies $a\leq b$ for $a,b\in \Lambda$. Suppose $a\not\leq b$ for some $a,b\in \Lambda$. Since $A$ is compactly generated there exists $m\in A$ such that $m\leq a$ and $m\not\leq b$. Then $b\in\mathcal{X}(m)$ and hence, by Proposition \ref{compact elem 1} we have $b\leq p$ fore some maximal element $p$ in $\mathcal{X}(m)$. Note that $p\notin d(b)$ and $a\not\leq p$ since $m\not\leq p.$ Thus $p\notin d(b)$ and $p\in d(a)$. This gives $d(a)\not\subseteq d(b).$
\end{proof}

\begin{obs}
In \cite[Section 2]{simmonssome} a more general situation is given.
\end{obs}

\vspace{1em}

\subsection{Module theoretic preliminaries}\label{subsecmod}

%Throughout this paper $R$ will be an associative ring with identity, not necessarily commutative. The word \emph{ideal}will mean two-sided ideal, unless explicitly stated the side (left or right ideal). All modules are unital and left modules. Let $M$ be an $R$-module, a submodule $N$ of $M$ is denoted by $N\leq M$, whereas we write $N<M$ when $N$ is a proper submodule of $M$. Recall that $N\leq M$ is called fully invariant submodule, denoted by $N\leq_{fi}M$, if for every endomorphism $f\in End_R(M)$, $f(N)\subseteq N.$ Denote by 
%\[\Lambda(M)=\{N\mid N\leq M\}\]
%\[\Lambda^{fi}(M)=\{N\mid N\leq_{fi}M\}\]

As it was mentioned in the Introduction, we want to translate some notions of rings to the module context. In order to do this, we will work in the category $\sigma[M]$ where $M$ is an $R$-module. The category $\sm$ is the full subcategory of $R\Mod$ consisting of all modules that can be embedded in a $M$-generated module. This category is more general than $R\Mod$ in the sense that if $M=R$ then $\sm=R\Mod$. It can be seen that $\sm$ is a category of Grothendieck \cite{wisbauerfoundations}. Many ring-theoretic aspects have been translated to modules in this category, see for example \cite{PepeGab}, \cite{PepeFbn}, \cite{PepeKrull}, \cite{maugoldie}, \cite{wisbauerfoundations}, \cite{wisbauermodules}, etc.

In \cite{BicanPr} is defined a product of modules as follows:

\begin{dfn}\label{pro}
Let $M$ and $K$ be $R$-modules. Let $N\leq M$. The product of $N$ with $K$ is defined as:
\[N_MK=\sum\{f(N)\mid f\in \Hom(M,K)\}\]
\end{dfn}

This product generalizes the usual product of an ideal and an $R$-module. For properties of this product see \cite[Proposition 1.3]{PepeGab}. In particular we have a product of submodules of a given module. 

Given a submodule $N$ of a module $M$, we will denote the least fully invariant submodule of $M$ containing $N$ by  $\overline{N}$. This submodule can be described as 
\[\overline{N}=N_M M.\]

Since we have a product, it is natural to ask for an annihilator. Next definition was given in \cite{beachy2002m}:

\begin{dfn}
Let $M$ and $K$ be $R$-modules. The \emph{annihilator} of $K$ in $M$ is defined as:

\centerline{$\annm(K)=\bigcap\{Ker(f)\mid f\in\Hom(M,K)\}$}
\end{dfn}

This annihilator is a fully invariant submodule of $M$ and it is the greatest submodule of $M$ such that $\annm(K)_MK=0$. 

%%%%%%%%%%%%
% Propiedades que satisface el anulador
%%%%%%%%%%%%%

Now we present two lemmas that will be needed in what follows.

\begin{lem}\label{anninter}
Let $M$ be projective in $\sm$ and $\{N_i\mid i\in I\}$ a family of modules in $\sm$. Then $\bigcap\annm(N_i)=\annm(\sum N_i)$.
\end{lem}

\begin{proof}
Denote $N=\sum_I N_i$. Since $N_i\leq N$ then $\annm(N)\leq \bigcap_I \annm(N_i)$. Now, let $f\colon M\to N$ be a non zero morphism and consider the canonical epimorphism $\rho\colon\bigoplus_I N_i\to N$. So, since $M$ is projective in $\sm$ there exists $g\colon M\to \bigoplus_I N_i$ such that $\rho g=f$.\vspace{-8pt}
\[\xymatrix{ & M\ar[d]^f \ar@{--{>}}[dl]_g \\ \bigoplus_I N_i\ar[r]_\rho & N}\]
Let $x\in\bigcap_I \annm(N_i)$, then $\pi_i(g(x))=0$ for all $i\in I$, where $\pi_i\colon\bigoplus_I N_i\to N_i$ are the canonical projections, hence $g(x)=0$. Thus $f(x)=\rho(g(x))=0$. This implies $\bigcap_I\annm(N_i)\leq \annm(N)$.
\end{proof}

\begin{lem}\label{rayita}
Let $M$ be projective in $\sm$. If $N$ is a submodule of $M$, then $\annm(\overline{N})=\annm(N)$.
\end{lem}

\begin{proof}
Let $N$ be a submodule of $M$. We always have $\annm(\overline{N})\leq\annm(N).$ On the other hand, using the associativity of the product in $M$, we have 
\[\annm(N)_M \overline{N}= \annm(N)_M N_M M=0.\]
\end{proof}

Other concepts related to the product that come up are primeness and semiprimeness.

\begin{dfn}[Definition 13 \cite{raggiprime}]
Let $N\leq M$ be a proper fully invariant submodule. $N$ is \emph{prime} in $M$ if whenever $L_MK\leq N$ with $L,K\leq_{fi} M$ then $K\leq N$ or $L\leq N$. We say that $M$ is a prime module if $0$ is prime in $M$.
\end{dfn}

\begin{dfn}[\cite{raggisemiprime}]
Let $N\leq M$ be a proper fully invariant submodule. $N$ is \emph{semiprime} in $M$ if whenever $L_ML\leq N$ with $L\leq_{fi} M$ then $L\leq N$. We say that $M$ is a semiprime module if $0$ is semiprime in $M$.
\end{dfn}

The product of submodules is neither associative nor  distributes sums from the right, in general. If we assume that $M$ is projective in $\sigma[M]$ then the product is associative and distributive over sums \cite[Proposition 5.6]{beachy2002m} and \cite[Lemma 1.1]{maustructure}. Moreover if $N,L\leq_{fi}M$ then $N_ML\leq_{fi}M$. See \cite[Remark 4.2]{mauquantale}.

Now, we give some properties of prime and semiprime submodules:

\begin{prop}\label{1.9}
Let $M$ be projective in $\sm$ and $P$ a fully invariant submodule of $M$. The following conditions are equivalent:
\begin{enumerate}[\rm(1)]		
	\item $P$ is prime in $M$.
	\item For any submodules $K$, $L$ of $M$ such that $K_ML\leq{P}$, then $K\leq P$ or $L\leq P$.
	\item For any submodules $K$, $L$ of $M$ containing $P$ and such that $K_ML\leq{P}$, then $K=P$ or $L=P$.
	\item $M/P$ is a prime module.
\end{enumerate}
\end{prop}

\begin{proof}
It follows from \cite[Proposition 1.11]{PepeGab}, \cite[Proposition 5.5]{beachy2002m} and \cite[Proposition 18]{raggiprime}.
\end{proof}

\begin{prop}\label{1.14}
Let $M$ be projective in $\sigma[M]$ and $N$ a fully invariant submodule of $M$. The following conditions are equivalent:
\begin{enumerate}[\rm(1)]
	\item $N$ is semiprime in $M$.
	\item For any submodule $K$ of $M$, $K_MK\leq{N}$ implies $K\leq{N}$.
	\item For any submodule $K\leq{M}$ containing $N$ such that $K_MK\leq{N}$, then $K=N$.
	\item $M/N$ is a semiprime module.
	\item If $m\in{M}$ is such that ${Rm}_M{Rm}\leq{N}$, then $m\in{N}$.
	\item $N$ is an intersection of prime submodules.
\end{enumerate}
\end{prop} 

\begin{proof}
$(1)\Leftrightarrow (2)\Leftrightarrow (3)\Leftrightarrow (4)$ It is analogous to Proposition \ref{1.9}.

$(1)\Leftrightarrow 5\Leftrightarrow (6)$ See \cite[Proposition 1.11]{maugoldie}.
\end{proof}

\section{The frame of semiprimitive submodules.}\label{sec3}

In this section, we  define primitive submodules and semiprimitive submodules. We wil prove that the set of all semiprimitive submodules and M is a spatial frame, furthermore , we will  see that this frame is  isomorphic to the frame of open set of the topological space $\mx(M).$  Since we have this frame, we also ask us how is given its point space, and then we give sufficient conditions for its point space to be precisely the primitive submodules.

\begin{dfn}
Let $M$ be an $R$-module and $N<M$. It is said that $N$ is a \emph{primitive submodule} if $N=\annm(S)$ for some $S\in\sm$ a simple module. We say that $M$ is a \emph{primitive module} if $0$ is a primitive submodule. \end{dfn}
Denote $\prt(M)=\{P< M\mid P\text{ is primitive }\}.$

\begin{obs}
If $0\neq M$ then there are simple modules in $\sigma[M]$. Consider $Rm\leq M$ with $0\neq m$, since $Rm$ is cyclic then it has maximal submodules, say $\mathcal{M}\leq Rm$, then $Rm/\mathcal{M}\in \sigma[M]$ is simple.
\end{obs}

\begin{prop}\label{primprime}
Let $M$ be projective in $\sm$. If $P< M$ is a primitive submodule then $P$ is prime in $M$.
\end{prop}

\begin{proof}
Let $N,L$ be fully invariant submodules of $M$ such that $N_ML\leq P$. Since $P$ is a primitive submodule of $M$ there exists $S\in\sm$ simple module such that $P=\annm(S)$, then $(N_ML)_MS=0$. Furthermore, $N_M(L_MS)=0$ because $M$ is projective in $\sm$. Since $L_MS\leq S$, we have that $L_MS=0$ or $L_MS=S$. If $L_MS=0$ then $L\leq P$. On the other hand, if $L_{M}S=S$ then $0=N_{M}(L_{M}S)=N_{M}S$. Thus $N\leq P$.
\end{proof}

\begin{obs}
It is possible that $M$ has no primitive submodules. Although, if $M$ has a maximal submodule $\mathcal{M}$, then $P=\annm(M/\mathcal{M})$ is a primitive submodule.
\end{obs}

\begin{dfn}
Let $M$ be an $R$-module. It is said that $M$ is \emph{coatomic} if every proper submodule is contained in a maximal submodule. 
\end{dfn}

\begin{ej} The following are examples of coatomic modules:
\begin{enumerate}[(a)]
	\item Finitely generated modules,
	\item semisimple modules,
	\item semiperfect modules,
	\item multiplication modules over a commutative ring,
	\item modules over a left perfect ring,
\end{enumerate}
\end{ej}

We will denote by $\mx(M)$ and $\mx^{fi}(M)$ the maximal elements in the lattices $\Lambda(M)$ and $\Lambda^{fi}(M)$ respectively.

\begin{lem}\label{maxprm}
Let $M$ be coatomic and projective in $\sm$. If $\mathcal{N}\in\mx^{fi}(M)$ then $\mathcal{N}$ is a primitive submodule.
\end{lem}

\begin{proof}
Since $M$ is coatomic, there exists a maximal submodule $\mathcal{M}$ such that $\mathcal{N}\leq \mathcal{M}$. Since $M$ is projective in $\sm$ and $\mathcal{N}$ is fully invariant, then $\mathcal{N}_M(M/\mathcal{M})=0$ by \cite[Proposition 1.3 and Proposition 1.8]{PepeGab}. Thus $\mathcal{N}\leq \annm(M/\mathcal{M}),$ but $\mathcal{N}\in \mx^{fi}(M)$, therefore $\mathcal{N}=\annm(M/\mathcal{M})$.
\end{proof}

\begin{obs}
The converse of the last Lemma is not true in general. For instance, consider the following examples:
\begin{enumerate}[\rm(1)]
\item The  ring $R$ constructed by Bergman and described in \cite[pp. 27]{chattersrings}. It can be seen that $R$ is a primitive ring and has a unique non zero two-sided ideal $U$. Therefore, $0$ is a primitive ideal but it is not a maximal ideal.
\item  Let $K$ be a field and $V$ a $K-$vector space with $dim_K(V)=\aleph_0$ and let $R:=End_K(V).$ Notice that $V$ is an $R-$module, $Ann_R(V)=0$ is not a maximal ideal in $R$ and $I=\{f\in R\mid dim_K(Im(f))\text{ is finite}\}$  is the only maximal ideal of $R.$ Let $f:V\to V$ such that $dim_K(Im(f))=1.$ Thus, $Rf$ is a projective simple $R-$module and $Ann_R(Rf)=0$ is not maximal. See \cite[Ex. 3.15, Ex. 4.8]{Lam}.
\end{enumerate}
\end{obs}

Now, we are interested in a particular class of submodules defined by the primitive submodules. We will prove that this class can be seen as the fixed points of a suitable operator on $\Lambda^{fi}(M)$.

\begin{dfn}
A submodule $N\leq M$ is called a \emph{semiprimitive submodule} if $N$ is an intersection of primitive submodules.
\end{dfn}

Let $M$ be an $R$-module and $\mx(M)=\{\mathcal{M}<M\mid \mathcal{M}\text{ is maximal }\}$. If $M$ is projective in $\sigma[M]$, then $\mx(M)$ is a topological space with open sets $\{\textit{m}(N)\mid N\leq_{fi}M\}$ where $\textit{m}(N)= \{\mathcal{M}\in \mx(M)\mid N\nleq \mathcal{M}\}$. See \cite{mauquantale}.

If $M$ is coatomic, then $\mx(M)\neq \emptyset$. Hence, we have an adjunction
 \vspace{-10pt}

\[\xymatrix@=15mm{\Lambda^{fi}(M)\ \ \ar@/^/[r]^{\textit{m}} & \mathcal{O}(\mx(M))\ \ \ar@/^/[l]^{\textit{m}_*}}\]

This adjuntion defines a multiplicative nucleus $\tau:=\textit{m}_*\circ\textit{m}$ on $\Lambda^{fi}(M)$ (see \cite[Theorem 3.21]{mauquantale}). 

\begin{obs}
Given $N\leq_{fi}M$, by \cite[Proposition 3.20]{mauquantale} $\tau(N)$ is the largest fully invariant submodule contained in $\bigcap\{\mathcal{M}\in \mx(M)|N\leq\mathcal{M}\}.$ 

Consider $\Lambda^{fi}(M)_{\mu}=\{N\in\Lambda^{fi}(M)\mid\tau(N)=N\}.$

%We will denote the fixed points of $\tau$ as
%\[SPm(M)=\{N\in\Lambda^{fi}(M)\mid\tau(N)=N\}.\]

\end{obs}

%\begin{prop}\label{tauprim}
%Let $M$ be projective in $\sm$ and $\mathcal{M}\in\mx(M)$. Then $\tau(\mathcal{M})$ is primitive submodule of $M$.
%\end{prop}
%
%\begin{proof}
%Since $\tau(\mathcal{M})\in\Lambda^{fi}(M)$, by \cite[Proposition 1.8]{PepeGab} $\tau(\mathcal{M})_M\left(M/\tau(\mathcal{M})\right)=0$. On the other hand, there exists an epimorphism $M/\tau(\mathcal{M})\to M/\mathcal{M}$, then by \cite{beachy2002m} $\tau(\mathcal{M})_M\left(M/\mathcal{M}\right)=0$. This implies that $\tau(\mathcal{M})\subseteq\annm(M/\mathcal{M})$. Since $\tau(\mathcal{M})$ is the largest fully invariant submodule contained in $\mathcal{M}$, $\tau(\mathcal{M})=\annm(M/\mathcal{M})$. Thus, $\tau(\mathcal{M})$ is a primitive submodule.
%\end{proof}

\begin{prop}
Let $M$ be an $R$-module, then $\prt(M)\subseteq \Lambda^{fi}(M)_{\mu}.$
\end{prop}

\begin{proof}
If $P\in\prt(M),$ then $P=\annm(S)=\bigcap\{Ker(f)\mid f\in\Hom(M,S)\},$
with $S$ a simple module. Thus, we can take only the nonzero morphisms in $\Hom(M,S)$, hence $P$ is an intersection of maximal submodules, therefore
\[\tau(P)\leq \bigcap\{\mathcal{M}\in \mx(M)\mid P\leq\mathcal{M}\}\leq P.\]
Since always $P\leq \tau(P)$, then $P=\tau(P)$. Hence, $\prt(M)\subseteq \Lambda^{fi}(M)_{\mu}$.
\end{proof}

\begin{prop}\label{ann}
Let $M$ be projective in $\sm$ and $K$ be a proper fully invariant submodule of $M$. Then 
\vspace{-4pt}
\[\bigcap\{\annm\left(M/\mathcal{M}\right)\mid K\leq\mathcal{M}\in \mx(M)\} = \bigcap\{\mathcal{M}\in \mx(M)\mid K\leq \mathcal{M}\}.\]
\end{prop}

\begin{proof}
Let $\mathcal{M}\in \mx(M)$ such that $K\leq\mathcal{M}$. Since  $\annm(M/\mathcal{M})\leq \mathcal{M},$ then \vspace{-5pt}
\[ \bigcap\{\annm\left(M/\mathcal{M}\right)\mid K\leq\mathcal{M}\in \mx(M)\} \leq \bigcap\{\mathcal{M}\in \mx(M)\mid K\leq \mathcal{M}\}.\]

Now, let $f\in \Hom(M,M/\mathcal{M})$ be nonzero and $A=\ker(f)$. Notice that $A$ is a maximal submodule of $M$. Since  $\frac{M}{A}\cong\frac{M}{\mathcal{M}}$, then $K_M(\frac{M}{A})=0$. Therefore $K\leq \annm(M/A)\leq A$. This implies that $\annm(M/\mathcal{M})$ is an intersection of maximal submodules containing $K$. Thus 
\[\bigcap\{\mathcal{M}\in \mx(M)\mid K\leq \mathcal{M}\}\leq  \bigcap\{\annm\left(M\mathcal{M}\right)\mid K\leq\mathcal{M}\in \mx(M)\}.\]
\end{proof}

\begin{thm}\label{rm}
Let $M$ be projective in $\sm$. Then 
\[\Lambda^{fi}(M)_{\mu}=\{N\in\Lambda^{fi}(M)\mid  N \text{ is a  semiprimitive submodule }\}\cup\{M\}\]
is an spatial frame. Moreover, \[\Lambda^{fi}(M)_{\mu}\cong\mathcal{O}(\mx(M))\] \noindent and  $K=\bigcap\{\annm\left(\frac{M}{\mathcal{M}}\right)\mid K\leq\mathcal{M}\in \mx(M)\},\,$  for every $K\in SPm(M).$
\end{thm}

\begin{proof}
Let $N$ be a semiprimitive submodule, then there exist $\{P_i\}_{i\in I}\subseteq \Lambda^{fi}(M)_{\mu}$ such that 
$N=\displaystyle\bigcap\{P_{i}\mid{i\in I}\}.$
Since every primitive submodule is an intersection of maximal submodules containing $N$, then 
\[N\leq \tau(N) \leq \bigcap\{\mathcal{M}\in \mx(M)\mid N\leq \mathcal{M}\}
 \leq \bigcap_I\{P_{i}\mid P_{i}\in \prt(M)\} =N.\]
Thus $\tau(N)=N$. Clearly, $M\in SPm(M).$

Now let $K$ be a proper submodule of $M$ such that $K=\tau (K)$. Since $K$ is fully invariant, $K_M(\frac{M}{\mathcal{M}})=0$, and $K\leq \annm(M/\mathcal{M})\leq \mathcal{M}$ for every maximal $\mathcal{M}$ containing $K$. By Proposition \ref{ann} we have that
\begin{equation*}
\begin{split}
K& \leq \bigcap\{\annm{\left({M}/{\mathcal{M}}\right)}\mid K\leq\mathcal{M}{\in} \mx(M)\}= \bigcap\{\mathcal{M}{\in} \mx(M)\mid K\leq \mathcal{M}\}.
\end{split}
\end{equation*} 

Each $\annm(\frac{M}{\mathcal{M}})$ is a primitive submodule and is fully invariant. Since $K=\tau(K)$ is the largest fully invariant submodule contained in $\bigcap\{\mathcal{M}\in \mx(M)\mid K\leq \mathcal{M}\}$, then $K=\bigcap\{\annm\left(\frac{M}{\mathcal{M}}\right)\mid K\leq\mathcal{M}\in \mx(M)\}$.

By \cite[Corolary 3.11]{mauquantale}, it follows that $\Lambda^{fi}(M)_{\mu}$ is a frame and by construction 
$\Lambda^{fi}(M)_{\mu}\cong\mathcal{O}(\mx(M)).$

%Now let $K$ be a proper submodule of $M$ such that $K=\tau (K)$. Let $A,\mathcal{M}\in \mx(M)$ such that $K\leq\mathcal{M}$ and $\frac{M}{A}\cong\frac{M}{\mathcal{M}}$. Since $K$ is fully invariant $K_M(\frac{M}{\mathcal{M}})=0$, so $K_M(\frac{M}{A})=0$. Therefore $K\leq \annm(M/A)\leq A$. Thus
%\[K\leq\bigcap\{\mathcal{M}\in \mx(M)|K\leq \mathcal{M}\}= \bigcap\{\annm\left(\frac{M}{\mathcal{M}}\right)|K\leq\mathcal{M}\in \mx(M)\}\]
%
%Each $\annm(\frac{M}{\mathcal{M}})$ is a primitive submodule and is fully invariant. Since $K=\tau(K)$ is the largest fully invariant submodule contained in $\bigcap\{\mathcal{M}\in \mx(M)|K\leq \mathcal{M}\}$, then $K=\bigcap\{\annm\left(\frac{M}{\mathcal{M}}\right)|K\leq\mathcal{M}\in \mx(M)\}$
\end{proof}

\begin{obs}
Inasmuch as Theorem \ref{rm}, for the remainder of the text, we will denote  by $SPm(M)$
the spatial frame of all semiprimitive submodules $\Lambda^{fi}(M)_{\mu}$ for a projective module $M$ in $\sigma[M].$
\end{obs}

\begin{prop}\label{prmprim}
Let $M$ be projective in $\sm$. Then 
\[\prt(M)\subseteq\pt(SPm(M))\subseteq \spec(M).\]
\end{prop}

\begin{proof}
Since $K_ML\leq K\cap L$ for all $K, L\in\Lambda^{fi}(M)$, every prime submodule is $\wedge$-irreducible. Then $\prt(M)\subseteq\pt(SPm(M))$.

Let $Q\in \pt(SPm(M))$ and $N, L\in\Lambda^{fi}(M)$ such that $N_ML\leq Q$. We can apply the nucleus $\tau$ and we get 
\[\tau(N)\cap\tau(L)=\tau(N\cap L)=\tau(N_ML)\leq \tau(Q)=Q\]
Since $Q$ is a point in $SPm(M)$ then $N\leq\tau(N)\leq Q$ or $L\leq\tau(L)\leq Q$. 
\end{proof}

After last Proposition is natural to ask when $\prt(M)=\pt(SPm(M))$. In general this equality does not hold, as the following example shows:

\begin{ej}
Consider $\mathbb{Z}$ the ring of integers. Then $\Lambda(\mathbb{Z})=\Lambda^{fi}(\mathbb{Z})$. Since $\mathbb{Z}$ is commutative
$\prt(\mathbb{Z})=\{p\mathbb{Z}\mid p\text{ is prime }\}=\mx(\mathbb{Z}).$

Now, note that $0=\bigcap\{p\mathbb{Z}\mid p\text{ is prime }\}\in SPm(\mathbb{Z})$. Since $\mathbb{Z}$ is uniform then $0$ is $\cap-$irreducible, that is, $0\in\pt(SPm(\mathbb{Z}))$ but it is not a primitive ideal. Thus $\pt(SPm(\mathbb{Z}))\neq \prt(\mathbb{Z})$.
\end{ej}

%When $pt(SPm(M))=\prt(M)$?
Denote by $M$-Simp a complete set of representatives of isomorphism classes of simple modules in $\sm$. Next propositions give sufficient conditions on a module $M$ in order to get the equality $\pt(SPm(M))=\prt(M)$.

\begin{prop}\label{simpf}
Let $M$ be projective in $\sm$ such that $M$-Simp is finite. Then $pt(SPm(M))=\prt(M)$.
\end{prop}

\begin{proof}
Let $K\in\pt(SPm(M))$. By  Theorem \ref{rm},
\[K=\bigcap\left\{\annm\left(M/\mathcal{M}\right)\mid K\leq\mathcal{M}\in \mx(M)\right\}.\]
By hypothesis $M$-Simp is finite then this intersection is finite. Every $\annm(M/\mathcal{M})$ is in $SPm(M)$ and $K$ is a $\wedge$-irreducible then, there exists $\mathcal{M}\in \mx(M)$ such that $K=\annm\left(\frac{M}{\mathcal{M}}\right)$. Thus $K$ is primitive.
\end{proof}

\begin{ej}
Examples of rings satisfying the last proposition are:
\begin{enumerate}[(a)]
\item Commutative artinian rings.
\item $R=\left(\begin{matrix}
\mathbb{Q} & 0 \\
\mathbb{R} & \mathbb{R}
\end{matrix}\right)$
\item Rings of upper triangular matrices with coefficients in a field. 
\end{enumerate}
\end{ej}

Now, recall that a ring $R$ is called $pm-$ring if every prime ideal is contained in a unique maximal ideal. In the study of $Spec(R)$ and $Max(R)$ for a commutative ring,  $pm-$rings had taken an important role, for instance see  \cite{sunrings}.
 So, as a generalization of the above concept, we introduce the following definition for modules 

\begin{dfn}\label{pmmodule}
Let $M$ be an $R$-module. It is said $M$ is a \emph{pm-module} if every prime submodule is contained in a unique maximal submodule.
\end{dfn}

\begin{prop}\label{pmpt}
Let $M$ be projective in $\sm$. If $M$ is a pm-module then 
\[\pt(SPm(M))=\prt(M)=\mx(M).\]
\end{prop}

\begin{proof}
We have that a primitive submodule $P$ is an intersection of maximal submodules, so if $M$ is a pm-module then $P$ has to be maximal. On the other hand, if $\mathcal{M}\in \mx(M)$ the $\annm(M/\mathcal{M})$ is primitive submodule and it is contained in $\mathcal{M}$. Thus $\mathcal{M}=\annm(M/\mathcal{M})$. Hence we have that $\mx(M)=\prt(M)$.

By Proposition \ref{prmprim} $\pt(SPm(M))\subseteq \spec(M)$. Now, let $Q\in\pt(SPm(M))$, then $Q$ is semiprimitive, that is, $Q$ is an intersection of primitive submodules, but the primitive submodules are the maximal submodules. Since $Q$ is prime in $M$, by hypothesis $Q$ is contained in a unique maximal submodule thus $Q$ is maximal. Therefore $\pt(SPm(M))\subseteq \mx(M)=\prt(M)$ and by Proposition \ref{prmprim} we have that $\prt(M)\subseteq\pt(SPm(M))$.
\end{proof}

Now we are going to show that the converse of Proposition \ref{simpf} is not true in general.

\begin{ej}
\begin{enumerate}[(a)]
\item Let $X$ be a discrete space. Let $C(X)$ be the (commutative) ring of all continuous functions from $X$ to $\mathbb{R}$ with the pointwise operations. By \cite[Theorem 2.11]{gillmanrings} the ring $C(X)$ is a pm-module (over itself). Thus by Proposition \ref{pmpt} $\mx(C(X))=\prt(C(X))=\pt(SPm(C(X)))$. 
The maximal ideals of $C(X)$ can be described as: $\mathcal{M}\leq C(X)$ is a maximal ideal if and only if $\mathcal{M}=\langle f_x\rangle$ for some $x\in X$, where $f_x$ is the function defined as
\[f_x(y)=0\text{ if } x=y\]
\[f_x(y)=1 \text{ if } x\neq y\]
Therefore, the ring $C(X)$ has as many maximal ideals as points in $X$. Thus, if $X$ in not finite, $C(X)$-Mod has infinitely many non isomorphic simple modules and $\prt(C(X))=\pt(SPm(C(X)))$. 
\item Consider the boolean ring $R:=\mathbb{Z}_2^{\aleph_0}$ and note that $\mathcal{M}\in Max(R)$ if and only if $\mathcal{M}$ is prime. Also, their factors are not isomorphic $R-$modules  and $|R-simp|=\aleph_0.$
\end{enumerate}
\end{ej}

\section{Sobriety}\label{sec4}

In this section we will study the \emph{sobriety} of the  spaces $\spec(M)$ and $\mx(M),$ where $M$ is an $R$-module projective in $\sm.$ 
We will see that in this case, $\spec(M)$ is sober. On the other hand, in general $\mx(M)$ is not sober and then we will give sufficient and necessary conditions for being that.

For properties of sober spaces and theory related to this topic we refer to the reader to \cite{simmonspoint}, \cite{johnstone1986stone} and \cite{picado2012frames}. 

Given $S$  a topological space and $X$  any subset of $S,$ we will denote $X^{-}$ and $X'$ closure of $X$ and the complement of $X$ in $S,$ respectively. 

\begin{dfn}\label{sob1}
A  nonempty subset $X$ of a topological space $S$ is \emph{closed irreducible} if it is closed, and satisfies \[(X\cap U)\text{ and } 
(X \cap V)\text{ are not empty} \Rightarrow X\cap(U\cap V)\text{ is not empty }\] for all $U, V\in\mathcal{O}(S)$.
\end{dfn}

\begin{dfn}\label{sob2}
A topological space $S$ is \emph{sober} if each closed irreducible subset $X$ has a unique \emph{generic point}, that is, there exists a unique point $s\in X$ such that $X=s^{-}$.
\end{dfn}

The full subcategory of sober spaces is reflective in the category of topological spaces, therefore for every space $S$ we can associate a sober space $\sob(S)$ and a continuous morphism $\zeta\colon S\rightarrow\sob(S)$, this assignation is called the \emph{soberfication} of $S$ (see \cite{simmonspoint}
). 

The space $\sob(S)$ consists of all closed irreducible subsets of $S$ and its topology is given by the family of sets:
\[\natural{U}=\{X\in\sob(S)\mid X\cap U\neq\emptyset\}\]
with $U\in\mathcal{O}(S)$; and the continuous map is given by $\zeta(s)=s^{-}$.

As examples of sober spaces we have the following two results.

\begin{thm}\label{sob3}
For each frame $A$ the point space $\pt(A)$ is sober.
\end{thm}

\begin{proof}
See \cite[Theorem 3.6]{simmonspoint}.
\end{proof}

\begin{obs}\label{sobpt}
Using this Theorem, it can be seen that for every topological space $S,$ the assignation $(-)'\colon\pt(\mathcal{O}(S))\to \sob(S),$ such that  $U\mapsto U',$
is a homeomorphism (see \cite{simmonspoint}).
\end{obs}

\begin{prop}\label{specsob}
Let $M$ be projective in $\sm$. Then $\spec(M)$ is sober. 
\end{prop}

\begin{proof}
$\spec(M)$ is a topological space with open sets $\{\mathcal{U}(N)\mid N\in\Lambda^{fi}(M)\}$ where $\mathcal{U}(N)=\{P\in \spec(M)\mid N\nleq P\}$. So the hull-kernel adjuntion gives a multiplicative nucleus $\mu$ on $\Lambda^{fi}(M)$ defined as
\[\mu(N)=\bigcap\{P\mid N\leq P\}.\]
Hence by \cite[Proposition 4.24]{mauquantale} we have that the set of fixed points of $\mu$ is
\[SP(M)=\{M\}\cup\{N\mid N\text{ is semiprime in } M\}.\]
By \cite[Proposition 4.29]{mauquantale}  
$Spec(M)=\pt(SP(M))\cong\pt(\mathcal{O}(\spec(M)).$
Hence, by Remark \ref{sobpt} 
\begin{equation*}
\begin{split}
\sob(\spec(M))& =\{\mathcal{U}(P)'\mid P\text{ is prime in }M\} =\{P^{-}\mid P\in \spec(M)\}.
\end{split}
\end{equation*}
Thus, by  \cite[Lemma 1.10]{simmonspoint}, we conclude that  $\spec(M)$ is sober.
\end{proof}

An immediately consequence of the above is:
\begin{cor}\label{specr}
For every ring $R$, the spectrum $\spec(R)$ is a sober space.
\end{cor}

\begin{prop}\label{dip}
Let $D$ be a principal ideal domain. Then $\sob(\mx(D))\cong \spec(D)$ as topological spaces. That is, the soberfication of $\mx(D)$ is $\spec(D)$.
\end{prop}

\begin{proof}
Inasmuch as  $D$ is a PID, every non zero prime ideal is maximal and the primitive ideals are the maximal ideals, so
$\prt(D)=\mx(D).$

 Since $D$ is a domain, 
$\pt(SPm(D))=\{0\}\cup\{\mathcal{M}<D\mid\mathcal{M}\text{ is maximal }\}.$
Under the isomorphism $SPm(D)\cong\mathcal{O}(\mx(D))$ we have that 
\[\pt(\mathcal{O}(\mx(D)))=\{\emptyset\}\cup\{\mathcal{U}(\mathcal{M})\mid\mathcal{M}\text{ is maximal }\}.\]

\noindent Thus,  using the fact that  if $\mathcal{M}$ is maximal then it is closed, we have that 
\begin{equation*}
\begin{split}
sob(\mx(D))=\{\mx(D)\}\cup\{\mathcal{M}^{-}\mid\mathcal{M}\text{ is maximal}\}\\=\{\mx(D)\}\cup\{\{\mathcal{M}\}\mid\mathcal{M}\text{ is maximal}\}.
\end{split}
\end{equation*}

Notice that if $U\in\mathcal{O}(\mx(D))$ then 
\[\natural{U}=\{X\in \sob(\mx(D))\mid X\cap U\neq\emptyset\}=\]\[\{\mx(D)\}\cup\{\{\mathcal{M}\}\in\sob(\mx(D))\mid \mathcal{M}\in U\}.\]
Because of the open sets of $\mx(D)$ are $\mathcal{U}(I)=\{\mathcal{M}\mid I\nsubseteq\mathcal{M}\}$ for some $I\leq D,$ it follows that  
$\natural(\mathcal{U}(I))=\{\mx(D)\}\cup\mathcal{U}(I).$

Applying the universal property of $\sob(\mx(D))$ (\cite[Theorem 3.9]{simmonspoint}), we get the following commutative diagram
\[\xymatrix{\mx(D)\ar@{^{(}-{>}}[r] \ar[d]_\zeta & \spec(D) \\ \sob(\mx(D))\ar@{--{>}}[ru]_{\exists !\varphi^{\#}} }\]

\noindent where $\varphi^{\#}(\mathcal{M})=\mathcal{M}$ and $\varphi^{\#}(\mx(D))=0$. Recall that a non trivial open set in $\spec(D)$ has the form:
\[\{P\in \spec(D)\mid I\nleq P\}=\{0\}\cup\{\mathcal{M}\in \mx(D)\mid I\nleq \mathcal{M}\}.\]
\end{proof}

From Proposition \ref{dip}, we can see that the space $\mx(M)$ is not sober in general. Next Theorem gives necessary and sufficient conditions for $\mx(M)$ to be sober.

\begin{lem}\label{tauprim}
Let $M$ be projective in $\sm$. Then $\mathcal{M}^-=m(\annm(M/\mathcal{M}))'$ for any $\mathcal{M}\in\mx(M)$.
\end{lem}

\begin{proof}
It is clear $\mathcal{M}\in m(\annm(M/\mathcal{M}))'$. On the other hand, let $m(K)'$ a closed set in $\mx(M)$ containing $\mathcal{M}$.Then $K\subseteq \mathcal{M}$. Since $M$ is projective in $\sm$, $K_MM/\mathcal{M}=0$. Hence $K\subseteq\annm(M/\mathcal{M})$. Thus $m(\annm(M/\mathcal{M}))'\subseteq m(K)'$.
\end{proof}

\begin{dfn}
Let $M$ be an $R$-module. It is said $M$ is quasi-duo if $\mx(M)=\mx^{fi}(M)$. A ring $R$ is called left quasi-duo if as left $R$-module is quasi-duo. 
\end{dfn}

Examples of left quasi-duo rings are the upper triangular matrices of $n$ by $n$ with coefficients in a field (See \cite[Proposition 2.1]{yuquasi}).   

\begin{lem}\label{qdmax}
Let $M$ be projective in $\sm$ such that $\mx(M)\neq\emptyset$. The following conditions are equivalent:
\begin{enumerate}[\rm(1)]
\item $M$ is a quasi-duo module.
\item $\mx(M)$ is $T_1$
\item $\mx(M)$ is $T_0$
\end{enumerate}
\end{lem}

\begin{proof}
(\textit{1})$\Rightarrow$(\textit{2}) Let $\mathcal{M}\in\mx(M)$. Since $\mathcal{M}$ is fully invariant, $\mathcal{M}=m(\mathcal{M})'$.

(\textit{2})$\Rightarrow$(\textit{3}) It is clear.

(\textit{3})$\Rightarrow$(\textit{1}) Let $\mathcal{M}\in\mx(M)$. Let $0\neq f:M\to M/\mathcal{M}$. Since $M/\ker(f)\cong M/\mathcal{M}$ then 
$\mathcal{M}^-=\textit{m}(\annm(M/\mathcal{M}))'=m(\annm(M/\ker (f)))'=ker(f)^-.$
But $\mx(M)$ is $T_0$, thus $\ker(f)=\mathcal{M}$. This implies that $\annm(M/\mathcal{M})=\mathcal{M}$.
\end{proof}

\begin{thm}
Let $M$ be projective in $\sm$. The following conditions hold: 
\begin{enumerate}[\rm(1)]
\item If $\mx(M)$ is sober then $\pt(SPm(M))=\prt(M)$.
\item If $M$ is a quasi-duo module and $\pt(SPm(M))=\prt(M)$ then $\mx(M)$ is sober.
\end{enumerate}
\end{thm}

\begin{proof}
Firstly, by Remark \ref{sobpt} it follows  that
\[\sob(\mx(M))=\{\textit{m}(K)'\mid K\in\pt(SPm(M))\}.\]

(\textit{1})  Assume that $\mx(M)$ is sober and let $\textit{m}(K)'\in \sob(\mx(M))$. Then  by Lemma \ref{tauprim},
$\textit{m}(K)'=\mathcal{M}^{-}=\textit{m}(\annm(M/\mathcal{M}))'$
for some $\mathcal{M}\in\mx(M)$. Therefore
$\textit{m}(K)=\textit{m}(\annm(M/\mathcal{M})).$

Since $m:SPm(M)\to \mathcal{O}(\mx(M))$ is a bijection %and  $\annm(M/\mathcal{M})$ $\in pt(SPm(M))$, 
it follows that  $K=\annm(M/\mathcal{M}),$ that is  $K$ is a primitive submodule.

(\textit{2}) By Lemma \ref{qdmax} $\mx(M)$ is $T_0$, hence it is enough to show that every closed irreducible is a point closure. Assume that  $K=\annm(S)$ with $S$ a simple module in $\sm$. Let $0\neq f:M\to S$. Then $\ker(f)\in\mx(M)$ and $K=\annm(M/\ker(f))$. Thus $m(K)'=\ker(f)^-$.
\end{proof}

\begin{cor}
If $R$ is a left quasi-duo ring such that $R$-Simp is finite then $\mx(R)$ is sober. In particular, $\mx(R)$ is sober for every commutative artinian ring $R$.
\end{cor}

\begin{proof}
It follows by Proposition \ref{simpf}.
\end{proof}

\begin{cor}
If $R$ is a pm-ring then $\mx(R)$ is a sober space.
\end{cor}

\begin{proof}
If follows by Proposition \ref{pmpt}.
\end{proof}

\section{The frame $\Psi(M)$}\label{sec5}

In this section, we introduce  $\Psi(M)$  which is a frame given by condition on annihilators. In fact, $\Psi(M)$ turns out to be a spatial frame. Firstly, we give an explicit definition of this frame and after that we get it as  the set of fixed point of  an operator in $\Lambda^{fi}(M)$. Also, we research about the case when  $\Psi(M)=\Lambda^{fi}(M)$ and we characterize the modules with this property.

Also on the definition of  $\Psi(M)$ arise the regular core for  an idiomatic-quantale. So, taking this in account, we give sufficient condition to get that the regular core of $\Lambda^{fi}(M)$ concides with  $\Psi(M).$

\begin{dfn}\label{psiM}
For a module $M$, set
 \[\Psi(M):=\{N\in\Lambda^{fi}(M)\mid \forall n\in N, [N+\annm(Rn)=M]\}.\] 
\end{dfn}

\begin{dfn}
We say that a module $M$ is self-progenerator in $\sm$ if $M$ is projective in $\sm$ and generates all its submodules.
\end{dfn}

\begin{prop}\label{psi1}
Let $M$ be self-progenerator in $\sigma[M]$. Then for every $N\in\Psi(M)$ we have that $N^{2}=N.$
\end{prop}

\begin{proof}
First, for each  $n\in N$ we have that \vspace{-4pt}
\[Rn{=}M_{M}Rn=[N+\annm(Rn)]_MRn=
N_{M}Rn+\annm(Rn)_MRn=N_{M}Rn.\] 
Then, by \cite[Lemma 2.1]{maustructure},

\centerline{$N^{2}=N_{M}N=N_{M}\sum Rn=\sum(N_{M}Rn)=\sum Rn=N.$}
\end{proof}

\begin{prop}\label{psi2}
Let $M$ be self-progenerator in $\sm$. If $N,L\in \Psi(M)$ then $N\cap L\in \Psi(M)$. Moreover, if $N\in \Psi(M)$ and $K\in\Lambda^{fi}(M)$ then $K\cap N=N_MK$.
\end{prop}

\begin{proof}
Let $N, L\in\Psi(M)$ and $n\in N\cap L$, then 
\begin{equation*}
\begin{split}
M=M_{M}M & =[N+\annm(Rn)]_{M}[L+\annm(Rn)]\\
& N_{M}L+N_{M}\annm(Rn)+\annm(Rn)^{2}+\annm(Rn)_ML\\
& \leq N_{M}L+\annm(Rn)\leq N\cap L+\annm(Rn).
\end{split}
\end{equation*}
Therefore $N\cap L\in\Psi(M)$.

Now let $N\in\Psi(M)$ and $K\in\Lambda^{fi}(M)$. By Proposition \ref{psi1}, if $n\in K\cap N$ then $Rn=N_MRn\leq N_MK\leq N\cap K$, thus $N\cap K=N_MK$.
\end{proof}

\begin{prop}\label{psi3}
Let $M$ be projective in $\sm$. If $\{N_{\alpha}\mid\alpha\in\EuScript{I}\}\subseteq\Psi(M)$ then $\sum\{N_{\alpha}\mid\alpha\in\EuScript{I}\}\in\Psi(M)$
\end{prop}

\begin{proof}
Let $N=\sum_\EuScript{I} N_\alpha$ and $x\in N$. Then $x=n_{\alpha_1}+\cdots n_{\alpha_k}$ with $n_{\alpha_i}\in N_{\alpha_i}$. By hypothesis $N_{\alpha_i}+\annm(Rn_{\alpha_i})=M$ for all $1\leq i\leq k$. Hence $N+\annm(Rn_{\alpha_i})=M$. Therefore, taking the product of $M$ $k$-times:
\begin{equation*}
\begin{split}
M=M_M\cdots {_MM} & =(N+\annm(Rn_{\alpha_1}))_M\cdots {_M(N+\annm(Rn_{\alpha_k}))}\\
& =N^k+\cdots+(\annm(Rn_{\alpha_1})_M\cdots {_M\annm(Rn_{\alpha_k})})\\
& \leq N+(\annm(Rn_{\alpha_1})_M\cdots {_M\annm(Rn_{\alpha_k})})\\
& \leq N+(\annm(Rn_{\alpha_1})\bigcap\cdots\bigcap\annm(Rn_{\alpha_k})).
\end{split}
\end{equation*}

On the other hand, since $Rx\leq Rn_{\alpha_1}+\cdots+ Rn_{\alpha_k}$ it follows that 
\[\annm(Rn_{\alpha_1}+\cdots +Rn_{\alpha_k})\leq \annm(Rx).\]
Thus, by Lemma \ref{anninter} 
$M\leq N+(\annm(Rn_{\alpha_1})\bigcap\cdots\bigcap\annm (Rn_{\alpha_k}))\leq N+\annm(Rx).$
\end{proof}

From Propositions \ref{psi2} and \ref{psi3} we have the following:

\begin{thm}\label{PSI}
Let $M$ be self-progenerator in $\sm$. Then $\Psi(M)$ is a spatial frame. 
\end{thm}  
\begin{proof}
From Proposition \ref{psi2}, $\Psi(M)$ is a distributive lattice. Therefore by Lemma \ref{disid} it is a frame. Thus $\Psi(M)$ is an spatial frame by Theorem \ref{spatiid}.
\end{proof}

%\begin{obs}\label{psi4}
%We have an adjoint situation: 
%
%\[\xymatrix{ \Psi(M)\ar@<-.7ex> @/   _1pc/[r]_--{\iota} &\Lambda^{fi}(M)\ar@<-.7ex>@/ _1pc/[l]_--{\psi} }\]
%
%that is, $\iota(N)\leq K\Leftrightarrow N\leq\psi(K)$ where $\iota$ is the inclusion and $\psi$ is the morphism given by $\psi(K)= \text{the largest element of } \Psi(M) \text{ included in } K$ , here $\psi$ is the right adjoint of $\iota$ ($\iota$ preserves $\bigvee$).
%\end{obs}

We can produce $\Psi(M)$ from another perspective. We can recover $\Psi(M)$ as the fixed point of a suitable operator in $\Lambda^{fi}(M)$. In this case such operator is not an  inflator but a deflator. 
\begin{dfn}\label{LerM}
For each $N\in \Lambda^{fi}(M),$ define 
\[Ler(N)=\{m\in M\mid N+\annm(Rm)=M\}.\]
\end{dfn}

%\begin{obs}
%Notice that $Ler(N)=\sum\{K\leq M \mid N+\Ann_M(K)= M\}$.
%\end{obs}

\begin{prop}
Let $M$ be projective in $\sm$. For each $N\in \Lambda^{fi}(M)$, $Ler(N)$ is a submodule of $M$.
\end{prop}
\begin{proof}
First, let $l,k\in Ler(N).$ Then $N+\annm(Rl)=M=N+\annm(Rk)$ and thus
$M=M_M M=(N+\annm(Rl))_M (N+\annm(Rk)).$ Using the fact that this product is distributive (because $M$ is projective), it follows that
%\[M=M_M M=(N+\annm(Rl))_M (N+\annm(Rk))=\]
\[M=N_MN+\annm(Rl)_MN+N_M\annm(Rk)+\annm(Rl)_M\annm(Rk).\]
%Notice that  $N_MN$, $\annm(Rl)_MN$  and $\annm(Rk)_MN$ are contained in $N.$ On the other hand,
%\[\annm(Rl)_M\annm(Rk)\subseteq \annm(Rl)\cap \annm(Rk)\subseteq \annm(R(l+k))\]
%by Lemma \ref{anninter}. 
Hence $M\subseteq N+\annm(R(l+k)).$ Therefore, $l+k\in Ler(N).$

Now, let $r\in R$ and $n\in Ler(N).$ Since $n\in Ler(N),$ we have that $N+\annm(Rn)=M.$ Inasmuch as $R(rn)\subseteq Rn$ it follows that $\annm(Rn)\subseteq \annm(R(rn)).$ Hence,
$M=N+\annm(Rn)\subseteq N+\annm(R(rn)), $ Thus $M=N+\annm(R(rn))$, that is, $rn\in Ler(N).$
\end{proof}

\begin{lem}\label{lercont}
Let $M$ be self-progenerator in $\sm$ and $N\in\Lambda^{fi}(M)$. Then $Ler(N)\leq N$.
\end{lem}

\begin{proof}
Let $N\in\Lambda^{fi}(M)$ and $m\in Ler(N)$. Then $M=N+\annm(Rm)$. Therefore,
$Rm=M_MRm=(N+\annm(Rm))_MRm=N_MRm\leq N.$
\end{proof}

\begin{prop}
Let $M$ be projective in $\sm$ and $N_1,...,N_n\in\Lambda^{fi}(M)$. Then
\[Ler(\bigcap_{i=1}^n N_i)=\bigcap_{i=1}^n Ler(N_i).\]
\end{prop}

\begin{proof}
Let $x\in \bigcap_{i=1}^n Ler(N_i)$, hence $M=N_i+\annm(Rx)$ for all $1\leq i\leq n$. 
\[M=M_MM=(N_i+\annm(Rx))_M(N_j+\annm(Rx))\]
\[\leq {N_i}_M{N_j}+\annm(Rx)\leq N_i\cap N_j + \annm(Rx)\]
Thus, $x\in Ler(N_i\cap N_j)$. Therefore $\bigcap_{i=1}^n Ler(N_i)\leq Ler(\bigcap_{i=1}^n N_i).$

We always have, $Ler(\bigcap_{i=1}^n N_i)\leq\bigcap_{i=1}^n Ler(N_i)$.
\end{proof}

\begin{prop}\label{fixpsi}
Let $M$ be self-progenerator in $\sm$. For each $N\in \Lambda^{fi}(M),$ $Ler(N)=N$ if and only if $N\in\Psi(M).$
\end{prop}

\begin{proof}
It follows from Definiton \ref{psiM}, Definiton \ref{LerM} and Lemma \ref{lercont}
\end{proof}

%\begin{prop}
%Let $M$ be self-progenerator in $\sigma[M]$. Then $\psi=Ler^{\infty}.$
%\end{prop}
%\begin{proof}
%Let $N\in\Lambda^{fi}(M)$. By Lemma \ref{lercont}, $Ler^\infty(N)\leq N$. Since 
%\[Ler^\infty(Ler^\infty(N))=Ler^\infty(N)\]
%then $Ler^\infty(N)\in \Psi(M).$ Thus $Ler^\infty(N)\leq\psi(N)$ 
%
%Now, we proceed by transfinite induction. By definition $\psi(N)\leq Ler(N)$. Let $\alpha=\beta+1$ be a successor ordinal and $n\in\psi(N)$. Hence $M=\psi(N)+\annm(Rn)\leq Ler^\beta(N)+\annm(Rn)$ by induction hypothesis. Thus $n\in Ler^\alpha(N)$.
%
%Now, let $\alpha$ a limit ordinal then $Ler^\alpha(N):=\bigwedge\{Ler^\beta(N)\mid {\beta<\alpha}\}$. Thus, by induction hypothesis we are done. Hence $\psi(N)\leq Ler^\infty(N)$. 
%\end{proof}

Now we are going to characterize the extreme case when $\Psi(M)=\Lambda^{fi}(M).$

\begin{dfn}\cite[page 89]{goodearlneumann}
Let $R$ be a ring. It is said that $R$ is \emph{biregular} if every cyclic two sided ideal is generated by a central idempotent.
\end{dfn}

\begin{prop}
A ring $R$ is biregular if and only if $R=RaR+Ann_R(RaR)$ for all $a\in R$.
\end{prop}

\begin{proof}
Let $a\in R$ and suppose $R=RaR+Ann_R(RaR)$. Then $1=sar+t$ with $r,s\in R$ and $t\in Ann_R(RaR)$. This implies that 
\[sar=(sar)sar+(sar)t=(sar)^2\]
so, $sar$ and $t$ are idempotents. Now, let $k\in R$. We have that
\[0=(sar)kt=(1-t)kt=kt-tkt\]
hence $kt=tkt$. In the same way $tk=tkt$, thus $t $ is central. This implies that $sar$ is central. Notice that $a=asar$, hence $RaR=RsarR$. Thus $R$ is biregular.

Reciprocally, if $a\in R$ there exists a central idempotent $e\in R$ such that $RaR=ReR$ and $R=ReR+R(1-e)R$. Since $e(1-e)=0$ then $R(1-e)\leq Ann_R(RaR)$. Thus $R=RaR+Ann_R(RaR)$. 
\end{proof}

In view of the above result, we introduce  the concept of birregular module as follows.

\begin{dfn} Let $M$ be an $R$-module. We say that $M$ is birregular if \[M=\overline{Rm}+\annm(\overline{Rm}),\] for every $m\in M.$ 

\end{dfn}

\begin{prop}Let $M$ be projective in $\sm$. $M$ is a birregular module if and  only if $\Psi(M)= \Lambda^{fi}(M).$

\end{prop}
\begin{proof}
First, suppose that $M$ is a birregular module. Let $N\in\Lambda^{fi}(M)$  and $n\in N.$ By Lemma \ref{rayita} 
 \[M=\overline{Rn}+\annm(\overline{Rn})= \overline{Rn}+\annm(Rn)\subseteq N+\annm(Rn).\]
 So $N\in \Psi(M)$ and thereby $\Psi(M)= \Lambda^{fi}(M).$

Conversely, assume that $\Psi(M)= \Lambda^{fi}(M).$ Then, for each $m\in M$ we have that $\overline{Rm}\in \Psi(M)=\Lambda^{fi}(M),$ consequently by Lemma \ref{rayita},   $M=\overline{Rm}+\annm(Rm).$ Hence, $M$ is birregular.
\end{proof}
%\begin{dfn}
%An idiom-quantale $A$ is {\it normal} if for each $a,b\in A$ such that $a\vee b=1,$ there exist $x,y\in A$ such that $a\vee x=b\vee y=1$ and $xy=0.$
%\end{dfn}
%
%Next we generalized Theorem 9.5 in ?.? for the module case.
%\begin{prop}
%The following statements are equivalent.
%\begin{enumerate}[(a)]
%\item $M$ is strong armonic.
%\item $\psi(N)\subseteq K\iff N\subseteq K,\,$ for each $N\in \Lambda^{fi}(M)$ and $K\in \mx^{fi}(M)$
%\item $\Lambda^{fi}(M)$ is normal.
%\item $\psi:\Lambda^{fi}(M)\to \Psi(M)$ has right adjoint.
%\item $\psi$ is $+$-preserving.
%\item $N+K=M\implies\psi(N)+\psi(K)=M,$ for each $N,K \in \Lambda^{fi}(M).$ 
%\item $N+K=M\implies  Ler(N)+ Ler(K)=M,$ for each $N,K \in \Lambda^{fi}(M).$ 
%\item $N+L=M\implies\psi(N)+L=M,$ for each $N\in \mx^{fi}(M),$ $L \in \Lambda^{fi}(M).$ 
%\end{enumerate}
%\end{prop}
%\begin{proof}
%
%\end{proof}

In what follows, we want to study a property of regularity in the frame $\Psi(M)$. 

In order to do that, we  recall some definitions and facts concerning to the regularity property. The background of this topics can see for example in \cite{johnstone1986stone} and \cite{simmons1989compact}. 

\begin{dfn}
Let $\Omega$ be a frame and $x,a\in\Omega$. It is said $x$ is rather below $a$, denoted by $x\eqslantless a$, if $a\vee \neg x=1$. Denote $\Omega^{\neg}$ the set of all $a\in\Omega$ with $a=\bigvee\{x\mid x\eqslantless a\}$.
\end{dfn}

\begin{obs}
It is well known that $\Omega^{\neg}$ is a subframe  of $\Omega$.
\end{obs}

\begin{dfn}
It is said that a frame $\Omega$ is \emph{regular} if $\Omega^{\neg}=\Omega$.
\end{dfn}
In general for a frame the subframe $\Omega^{\neg}$ not need to be regular there is an enhancement to this situation. 
We can associate to any idiomatic-quantale $A$, a regular frame $A^{reg}$ called \emph{the regular core} of $A$ \cite{simmons1989compact}. The construction of this regular frame is as follows:

Denote by $A^{r}$ the set of fixed points of the operator $r\colon A\rightarrow A$, given by $r(a)=\bigvee\{x\in A\mid a\vee x^{r}=1\}$ where $x^{r}=\bigvee\{y\mid yx=0\}$. Now, $A^{r}$ has his own $r$ deflator hence we can consider $A^{r(2)}:=(A^{r})^{r}$. Inductively, it is defined:
\[A^{r(0)}:=A\;\; A^{r(\alpha+1)}:= A^{r(\alpha)r}\;\; A^{r(\lambda)}:=\bigcap\{A^{r(\alpha)}\mid\alpha<\lambda\}\] 
for each non-limit ordinal $\alpha$ and limit ordinal $\lambda$ respectively. This chain is decreasing, therefore by a cardinality argument it eventually stabilizes in some ordinal. Let us denote the least of those ordinals by $\infty$ and $A^{reg}:=A^{r(\infty)}$.

In  \cite[Theorem 3.4]{simmons1989compact} it is proved that $A^{reg}$ is a regular frame and every regular subframe of $A$ is contained in it. 

\begin{prop}
Let $M$ be projective in $\sm$ and denote $A=\Lambda^{fi}(M)$. Then the deflator $r:A\to A$ is $Ler$ and $A^r=\Psi(M)$.
\end{prop}

\begin{proof}
Denote $A=\Lambda^{fi}(M)$ and let $N\in A$. So, $r(N)=\sum\left\{K{\in} A\mid N{+}{K^r}{=}M\right\}$ where $K^r=\sum\{L\in A\mid L_MK=0\}=\annm(K)$. Hence, \[r(N)=\sum\{K\in A\mid N+\annm(K)=M\}.\]

By definition, $Ler(N)=\{m\in M\mid N+\annm(Rm)=M\}$. Then
\[Ler(N)\subseteq\sum\{K\in A\mid N+\annm(K)=M\}.\]
On then other hand, if $K\in A$ is such that $M=N+\annm(K)$ then for every $k\in K$, $N+\annm(Rk)=M$. Hence $K\subseteq Ler(N)$. Thus $r(N)=Ler(N)$. This implies that $A^r=\Psi(M)$ by Proposition \ref{fixpsi}.
\end{proof}

In what follows, we want to study the case when $\Psi(M)$ is the regular core of $\Lambda^{fi}(M)$. At this moment, we just have sufficient conditions for this to happen.

\begin{thm}\label{negpsi}
Let $M$ be projective in $\sm$. If $Ler(\Ann_{M}(K))=\Ann_{M}(K)$ for all $K\in\Lambda^{fi}(M)$ then $\Psi(M)=(\Lambda^{fi}(M))^{reg}$.
\end{thm} 

\begin{proof}
Denote $A=\Lambda^{fi}(M)$. Let us show that $(A^r)^r=A^r$. Consider $r:A^r\to A^r.$ For $N\in A^r$ $r(N)=\sum\{K\in A^r\mid N+K^r=Ler(M)\}$
where $K^r=\sum\{L\in A^r\mid L_MK=0.\}$

Notice that $Ler(M)=M$. By hypothesis $Ler(\annm(K))=\annm(K)$ for all $K\in A$. Hence $\annm(K)\in A^r$ for all $K\in A^r$. Thus \[r(N)=\sum\{K\in A^r\mid N+\annm(K)=M\}.\]

Let $B\in A$ such that $N+\annm(B)=M$. By Proposition \ref{psi2},
\[B_M\annm(B)\subseteq B\cap\annm(B)=\annm(B)_MB=0.\]
Hence, $B\subseteq\annm(\annm(B))$. So,
\[M=N+\annm(B)\subseteq N+\annm(\annm(\annm(B))).\]
By hypothesis, $\annm(\annm(B))\in A^r$ then,
\begin{equation*}
\begin{split}
N& =\sum\{B\in A\mid N+\annm(B)=M\}\\
& \subseteq\sum\{\annm(\annm(B))\in A^r\mid N+\annm(\annm(\annm(B)))=M\}\\
& \subseteq r(N).
\end{split}
\end{equation*}
Thus, $r(N)=N$ and $(A^r)^r=A^r$.
\end{proof}

%With the following results, we want to characterize a class of modules which satisfy Theorem \ref{negpsi}. 
%
\begin{lem}\label{semi}
Let $M$ be projective in $\sm$ and semiprime , if \[M=N\oplus L\] with $N\in\Lambda^{fi}(M)$ then $L\in\Lambda^{fi}(M)$
\end{lem}

\begin{proof}
Since $N_{M}L\subseteq N\cap L=0,$ then $N_{M}L=0.$
Now $M$ is semiprime then $N_{M}L=0=L_{M}N$, thus $L_{M}M=L_{M}N\oplus L^{2}=L^{2}\leq L$ then $L=L_{M}M$.
\end{proof}

\begin{cor}\label{Semiber}
Let $M$ be projective in $\sm$ and  semiprime. If $\annm(N)$ is a direct summand of $M$ for all $N\in\Lambda^{fi}(M)$ then $\Psi(M)=(\Lambda^{fi}(M))^{reg}$.
\end{cor}

\begin{proof}
Let $N\in\Lambda^{fi}(M)$ then $M=\Ann_{M}(N)\oplus L$ by Lemma \ref{semi} we know that $L\in\Lambda^{fi}(M)$ thus $L\subseteq \Ann_{M}(\Ann_{M}(N))$ then 

\[\Ann_{M}(\Ann_{M}(N))=\Ann_{M}(\Ann_{M}(N))\cap M=\]\[ \Ann_{M}(\Ann_{M}(N))\cap (\Ann_{M}(N)\oplus L)=L.\]

\noindent  Thus $M{=}\Ann_{M}(N)\oplus \Ann_{M}(\Ann_{M}(N)).$ Then $\Ann_{M}(N)\leq Ler(\Ann_{M}(N))$. Thus $\Ann_{M}(N)= Ler(\Ann_{M}(N))$. By Theorem \ref{negpsi}, $\Psi(M)=(\Lambda^{fi}(M))^{reg}$.
\end{proof}

\begin{obs}
In \cite[Definition 3.2]{rizvibaer} are defined \emph{quasi-Baer modules}. Given $N\in\Lambda^{fi}(M)$, $\Hom_R(M,N)$ is a two-sided ideal of $\End_R(M)$. Then by \cite[Remark 3.3]{rizvibaer} in every quasi-Baer module $M$, $\annm(N)$ is a direct summand for all $N\in\Lambda^{fi}(M)$.
\end{obs}

\section{Acknowledgements}
This investigation is the result of a two-semester research seminar about the notes \cite{simmonssome}, the authors are thankful to Professor Harold Simmons who kindly provide to us that set of notes and many other interesting subjects around this.

The authors are very thankful to the referee for his/her comments on this paper that helped to improve it. 

This work was supported by the grant UNAM-DGAPA-PAPIIT IN10057.

\bibliographystyle{amsalpha}

\bibliography{biblio}

\footnotesize
 \vskip2mm
\noindent	 \textsc{Mauricio Medina B\'arcenas$^{\ast}$,  LorenaMorales Callejas$^{\star}$, Angel Zald\'ivar Corichi$^{\dag}$.}\\
Instituto de Matem\'aticas, Universidad Nacional Aut\'onoma de M\'exico, \'area de Investigaci\'on Cient\'ifica, Circuito Exterior, C.U., 04510, M\'exico, D.F., M\'exico.\vspace{5pt}

\noindent \textsc{Martha Lizbeth Shaid Sandoval Miranda$^{(\ddag)}$.}\\
Facultad de Ciencias, Universidad Nacional Aut\'onoma de M\'exico.\\
Circuito Exterior, Ciudad Universitaria, 04510, M\'exico, D.F., M\'exico.
\begin{itemize}
\item[ e-mails:]  $\ast$ mmedina@matem.unam.mx 
\item[\hspace{33pt}] $\star$ lore.m@ciencias.unam.mx
 \item[\hspace{33pt}]  $\dag$ zaldivar@matem.unam.mx 
\item[\hspace{33pt}] $\ddag$ marlisha@ciencias.unam.mx
\end{itemize}

\end{document}